\documentclass[a4paper,10pt]{amsart}

\usepackage[utf8]{inputenc}
\usepackage[english]{babel}

\usepackage{amsmath,amsthm,amsfonts,amssymb}

\usepackage[
    pagebackref = true,
    colorlinks = true,
    pdfstartview = FitH,
    pdfpagemode = UseNone,
    pdfauthor    = {Giacomo Cherubini Pietro Mercuri},
    pdftitle     = {Parity of 4-regular and 8-regular partition functions},
    pdfkeywords  = {regular partitions congruences modular forms dedekind eta},
    pdfcreator   = {Pdflatex},
    linktoc = none,
    bookmarks = false
    ]{hyperref}

\usepackage{cleveref}
\crefformat{equation}{(#2#1#3)}

\numberwithin{equation}{section}
\theoremstyle{plain}
\newtheorem{theorem}[equation]{Theorem}
\newtheorem{lemma}[equation]{Lemma}

\def\ZZ{\mathbb{Z}}
\def\QQ{\mathbb{Q}}
\def\CC{\mathbb{C}}

\def\HH{\mathcal{H}}

\def\Ocal{\mathcal{O}}

\title{Parity of 4-regular and 8-regular partition functions}

\author{Giacomo Cherubini}
\author{Pietro Mercuri}

\address{
         Charles University,
         Faculty of Mathematics and Physics,
         Department of Algebra,
         Sokolov\-sk\'a 83, 18600 Praha~8,
         Czech Republic
        }

\address{
      Istituto Nazionale di Alta Matematica ``Francesco Severi'',
      Research Unit Dipartimento di Matematica ``Guido Castelnuovo'',
      Sapienza Universit\`a di Roma, Piazzale Aldo Moro 5, I-00185, Roma
      }

\email{
    cherubini@karlin.mff.cuni.cz\\
    cherubini@altamatematica.it
    }

\address{
        Sapienza  Universit\`{a} di Roma, Department of SBAI, Rome, Italy
        }

\email{
    mercuri.ptr@gmail.com
    }

\date{\today}


\begin{document}

\begin{abstract}
We give a complete characterization of the parity of $b_8(n)$, the number of $8$-regular partitions of $n$. Namely, we prove that $b_8(n)$ is odd or even depending on whether or not we have the factorisation $24n+7=p^{4a+1}m^2$, for some prime $p\nmid m$ and $a\ge 0$.
\end{abstract}

\maketitle

\section{Introduction}

Partition functions are very natural and elementary objects,
being defined as the number of ways we can decompose a non-negative integer as the sum of positive integers, possibly with some constraints. In this paper we focus on $\ell$-regular partitions, i.e., decompositions where there is no summand divisible by $\ell$. The total number of $\ell$-regular partitions of an integer~$n$ is denoted by $b_\ell(n)$.

What is less trivial about partitions is that their generating function
has connections to the rather advanced theory of modular forms
and in particular to the Dedekind $\eta$ function (see \Cref{sec:modform}).
In \Cref{rem:prevref} we show that, when $\ell$ is a prime power, say $\ell=p^j$, the generating function of $b_{p^j}(n)$
can be essentially identified with $\eta^r$ modulo $p$,
for some positive even integer $r$ such that $\eta^r$
is \emph{lacunary}, which means that the set of its non-vanishing Fourier coefficients has density zero.
In particular, the $p$-divisibility of $b_{p^j}(n)$ can be studied by looking at the coefficients of $\eta^r$.

The connection between lacunary powers of $\eta$ and $b_{p^j}(n)$
requires $p^j-1$ to divide the exponent $r$. Since Serre in \cite{Ser85} showed that $\eta^r$ is lacunary if and only if $r\in\{2,4,6,8,10,14,26\}$, this implies that $p^j \in \{2,3,4,5,7,8,9,11,27\}$.
It turns out that all cases appeared in the literature: \cite{OP00,LP01,DP09,Abi22}, except for $p^j=8$. We prove therefore the following.

\begin{theorem}\label{thm:b8}
Let $n$ be a positive integer and let $b_8(n)$ denote
the number of $8$-regular partitions of $n$.
Then $b_8(n)$ is odd if and only if
\begin{equation}\label{2911:eq001}
24n+7 = p^{4a+1}m^2,
\end{equation}
for some prime $p\nmid m$ and some $a\ge 0$.
\end{theorem}
Note that, if \Cref{2911:eq001} holds, then we must have $p\equiv 7\pmod{24}$.
Curiously, the same factorisation condition appears in \cite[Theorem 1.4 (1)]{BL22} to characterise the behaviour modulo~$4$ of certain overpartitions with additional constraints, although we do not see a direct connection with $8$-regular partitions; also, our proof uses different techniques 
from the one in \cite{BL22}.

We point out that previous works where $p^j$-regular partitions have been studied by means of lacunary powers of the type $\eta^r$ always assumed that $r$ equals $p^j-1$. Relaxing this condition to the divisibility of the former by the latter, as explained in \Cref{rem:prevref}, allows us to obtain the missing case $p^j=8$.

In \Cref{sec:b4} we also show how the connection with lacunary powers of $\eta$ can be used to prove a characterization of the parity for $b_4(n)$, which is not contained in \cite{OP00,LP01,DP09,Abi22} but appeared before in the literature (see e.g.,~\cite[Table 1, Entry 2]{BBG87}) with different proofs. We include ours for the sake of completeness.

\begin{theorem}\label{thm:b4}
Let $b_4(n)$ denote the number of $4$-regular partitions of $n\in\ZZ_{\ge 0}$.
Then $b_4(n)$ is even if and only if $8n+1$ is not a square.
\end{theorem}

The paper is organised as follows:
in \Cref{sec:modform} we recall some well-known results about modular forms and describe the main tool in our argument (\Cref{rem:prevref}). In \Cref{sec:b4} we characterise the parity of $b_4(n)$;
then in \Cref{sec:b8} we treat $b_8(n)$.

\subsection*{Acknowledgements}
G.C.~received support by
the Czech Science Foundation GACR, grant 21-00420M,
the project PRIMUS/20/SCI/002 from Charles University,
and the Charles University Research Centre program UNCE/SCI/022.
This work began during a visit of P.M.~to Charles University,
which we thank for the support and the hospitality.
G.C.~is a Researcher at INdAM.

\section{Preliminaries}\label{sec:modform}

We begin by explaining how to connect regular partition
functions to modular forms
and in particular to powers of the Dedekind eta function
\[
\eta(z) = q^{1/24}\prod_{m=1}^{\infty} (1-q^m),
\]
where $q=e^{2\pi iz}$ and $z\in\HH=\{z\in\CC:\mathrm{Im}(z)>0\}$.
It is well known that $q^{1/24}\eta(z)^{-1}$ gives the generating function
of classical partitions. As for regular partitions, we use instead
the following observation.

\begin{lemma}\label{rem:prevref}
Let $p$ be a prime and let $r$ and $s$ be positive integers
such that $r=(p^j-1)s$ for some $j\geq 1$. Then
\[
\left(\sum_{n=0}^{\infty} b_{p^j}(n)q^{n+(p^j-1)/24}\right)^s
\equiv
\eta(z)^{r} \pmod{p}.
\]
\end{lemma}

\begin{proof}
First, using that $r+s=p^js$, we have the congruence
\[
(1-q^m)^r
= \frac{(1-q^m)^{r+s}}{(1-q^m)^s}
\equiv \frac{(1-q^{p^jm})^s}{(1-q^m)^s} \pmod{p}.
\]
If we multiply the term on the right over $m\geq 1$, we obtain the $s$-th power
of the generating function of $p^j$-regular partitions. Indeed, we have
\[
\prod_{m=1}^{\infty}\frac{1-q^{p^jm}}{1-q^m}
=
\sum_{n=0}^{\infty} b_{p^j}(n)q^n.
\]
Therefore, combining the two equations above, we deduce that
\[
\eta(z)^r = q^{r/24}\prod_{m=1}^{\infty}(1-q^m)^r
\equiv
\left(\sum_{n=0}^{\infty} b_{p^j}(n)q^{n+\frac{r}{24s}}\right)^s\pmod{p}.
\]
Using again $r=(p^j-1)s$ to simplify the exponent of $q$ we obtain the lemma.
\end{proof}

In particular, $(r,s)$ equal to $(6,2)$ and $(14,2)$ correspond to $p^j=4$ and $8$, respectively.
\Cref{rem:prevref} implies then
\begin{equation}\label{1411:eq001}
\eta(z)^r\equiv \left(\sum_{n=0}^{\infty} b_{2^j}(n)q^{n+(2^j-1)/24}\right)^2\equiv \sum_{n=0}^{\infty} b_{2^j}(n)q^{2n+(2^j-1)/12} \pmod{2}.
\end{equation}
As mentioned in the introduction, when $r\in\{2,4,6,8,10,14,26\}$,
we have finitely many possibilities for $p^j$,
all of which have been studied except for $p^j=4,8$.

The limitation on the values of $r$ comes from the fact that $\eta^r$
can be written as a linear combination of 
modular forms with complex multiplication (CM),
associated to Hecke characters on $\QQ(\sqrt{-1})$ or $\QQ(\sqrt{-3})$,
in these favourable cases only
(for the definition of modular forms with CM and their relations with Hecke characters,
see \cite[Section~3]{Rib77}).
More precisely, $\eta^r$ is a scalar multiple of a cusp form with CM
when $r\in\{2,4,6,8\}$,
the powers given by $r\in\{10,14\}$ are linear combinations of two forms
and $r=26$ requires four forms (see \cite[Sections~2.1--2.7]{Ser85}).

We end this section by recalling properties of the Fourier coefficients of an eigenform. In combination with \Cref{1411:eq001}, these are used in the next two sections to determine the parity of $b_4$ and $b_8$.
Let $N$ and $k$ be positive integers.
We denote by $\mathcal S_k(\Gamma_1(N),\chi)$ the $\CC$-vector space
of cusp forms of weight $k$ invariant under the action of $\Gamma_1(N)$
on which $\Gamma_0(N)$ acts via the Dirichlet character $\chi$ of modulo $N$.
For more details see e.g.~\cite[Section~4.3, p.~119]{DS05}.
Let $f$ be a normalised eigenform of $\mathcal S_k(\Gamma_1(N),\chi)$ with $q$-expansion
\[
f(z) = \sum_{n=1}^{\infty} a(n)q^n,
\]
where $q=e^{2\pi iz}$ and $z\in\HH$. The Fourier coefficients $a(n)$ are multiplicative, i.e.,
\begin{equation}\label{eq:mult}
a(nm)=a(n)a(m),
\end{equation}
whenever $\gcd(n,m)=1$ and on prime powers we have the recursion
\[
a(p^j) = a(p)a(p^{j-1})-\chi(p)p^{k-1}a(p^{j-2}),
\]
for $j\geq 2$ (see e.g.~\cite[Proposition~5.8.5]{DS05}).
By induction on $j$, we can write the above in closed form in terms of $a(p)$.
Thus, for every prime $p$ and $j\geq 0$, we have
\begin{equation}\label{0810:tpbeta}
a(p^j) = \sum_{r=0}^{\lfloor j/2\rfloor} (-1)^r \binom{j-r}{r} \chi(p)^r p^{(k-1)r} a(p)^{j-2r}.
\end{equation}

\section{Parity of $4$-regular partition function}\label{sec:b4}

We now focus on the proof of \Cref{thm:b4}.
By \cite[Section~2.3]{Ser85}, we have
\[
\eta^6(4z) = \varphi_{K,c}(z),
\]
where $\varphi_{K,c}\in\mathcal S_3(\Gamma_1(16),\chi)$,
with the Dirichlet character $\chi$ given by
\[
\chi(n)=\begin{cases}
(-1)^{(n-1)/2}, & \text{if $n$ is odd},\\
0, & \text{if $n$ is even}.
\end{cases}
\]
The form $\varphi_{K,c}$ is the normalised eigenform
associated to the Hecke character $c$ on $K=\QQ(i)$ defined as follows:
if $\mathfrak{a}$ is an ideal in $\mathcal{O}_K=\ZZ[i]$, with generator $\alpha$
such that $\alpha\equiv 1\pmod{2\mathcal{O}_K}$, we set
\[
c(\mathfrak{a}) = \alpha^2.
\]
With this notation, the cusp form $\varphi_{K,c}$ can be written as
\begin{equation}\label{eq:heckechar4}
\varphi_{K,c}(z) = \sum_{\mathfrak{a}} c(\mathfrak{a}) q^{\mathrm{Norm}(\mathfrak{a})},
\end{equation}
where the sum runs over non-zero ideals in $\Ocal_K$.
Note that $\varphi_{K,c}$ is listed with label 16.3.c.a on \cite{lmfdb}.

Next, by \Cref{1411:eq001}, we have
\[
\eta(4z)^6 \equiv \sum_{n=0}^{+\infty} b_4(n) q^{8n+1} \pmod{2}.
\]
It follows that
\begin{equation}\label{eq:b4an}
b_4(n) \equiv a(8n+1) \pmod 2,
\end{equation}
where $a(n)$ is the $n$-th Fourier coefficient of $\varphi_{K,c}(z)$.
By multiplicativity of the Fourier coefficients,
see \Cref{eq:mult}, it suffices to study $a(n)$ on prime powers.

\begin{lemma}\label{1411:lemma2}
If $p$ is a prime such that $p\equiv 1\pmod{4}$, then
\[
a(p^j)\equiv
\begin{cases}
0 \pmod{2}, &\text{if $j$ is odd},\\
1 \pmod{2}, &\text{if $j$ is even}.
\end{cases}
\]
\end{lemma}

\begin{proof}
Since $p\equiv 1\pmod{4}$, the prime $p$ splits in $K$ and we decompose it as $p=(x+iy)(x-iy)$ for $x,y\in\mathbb{Z}$. Without loss of generality, we can assume that $x$ is odd and $y$ is even. Therefore, by \Cref{eq:heckechar4} we have
\[
a(p) = (x+iy)^2 + (x-iy)^2 = 2(x^2-y^2) \equiv 0\pmod{2}.
\]
Next, looking at prime powers, by \Cref{0810:tpbeta} we have
\[
a(p^j)= \sum_{r=0}^{\lfloor j/2\rfloor} (-1)^r\binom{j-r}{r}\chi(p)^rp^{2r}a(p)^{j-2r}.
\]
But $p$ is odd, $\chi(p)=1$ and $a(p)$ is even, which implies that all summands are even
except possibly the last one, corresponding to $r=\lfloor j/2\rfloor$.
Distinguishing on the parity of $j$, we deduce that if $j$ is odd,
then $a(p^j)$ is even; whereas if $j$ is even, then $a(p^j)$ is odd.
\end{proof}

\begin{lemma}\label{1411:lemma1}
If $p$ is a prime such that $p\equiv 3\pmod{4}$, then
\[
\begin{cases}
a(p^j)=0, &\text{if $j$ is odd},\\
a(p^j)\equiv 1\pmod{2}, &\text{if $j$ is even}.
\end{cases}
\]
\end{lemma}

\begin{proof}
Since $p\equiv 3\pmod{4}$, the prime $p$ is inert in $K$ and therefore,
when $j$ is odd, there are no ideals of norm $p^j$, so $a(p^j)=0$ by \Cref{eq:heckechar4}.
When $j$ is even, there is only one ideal with norm $p^{j}$,
namely $p^{j/2}\mathcal{O}_K$, which is generated by $p^{j/2}$.
Applying \Cref{eq:heckechar4} again, we obtain
\[
a(p^j) = (p^{j/2})^2 = p^j \equiv 1 \pmod{2}.
\]
\end{proof}

\begin{proof}[Proof of \Cref{thm:b4}]
Let
\[
8n+1 = \prod_{p} p^{\alpha_p}
\]
be the prime factorisation of the odd number $8n+1$.
Hence $p=2$ never occurs.
By multiplicativity of the Fourier coefficients,
see \Cref{eq:mult}, we have
\[
a(8n+1) = \prod_{p} a(p^{\alpha_p}).
\]
By \Cref{1411:lemma2} and \Cref{1411:lemma1}, we see that as soon as
one of the exponents $\alpha_p$ is odd, we have an even factor and $a(8n+1)$ is even.
If instead $\alpha_p$ is even for all primes, then $a(8n+1)$ is odd.
The theorem follows from this and \Cref{eq:b4an}.
\end{proof}

\section{Parity of $8$-regular partition function}\label{sec:b8}

In this section we prove \Cref{thm:b8}.
By \cite[Section~2.6]{Ser85}, we have
\begin{equation}\label{0810:eq001}
\eta^{14}(12z) = \frac{1}{720\sqrt{-3}}(\varphi_{K,c_{+}}(z)-\varphi_{K,c_{-}}(z)),
\end{equation}
where $\varphi_{K,c_{\pm}}\in\mathcal S_7(\Gamma_1(144),\chi)$,
with $\chi$ being the Dirichlet character given by
\[
\chi(n)=\begin{cases}
(-1)^{(n-1)/2}, & \text{if }\gcd(n,144)=1,\\
0, & \text{otherwise}.
\end{cases}
\]
The forms $\varphi_{K,c_\pm}$ are the normalised eigenforms
associated to the Hecke characters $c_\pm$ on $K=\QQ(\sqrt{-3})$ of conductor
$\mathfrak{f}=4\sqrt{-3}\Ocal_K$ and defined as follows.
Let $\mathfrak{a}$ be an ideal in $\Ocal_K=\ZZ[\frac{1+\sqrt{-3}}{2}]$
coprime to $\mathfrak{f}$ and let $\alpha\in\Ocal_K$ be the unique generator of $\mathfrak{a}$ such that
\[
\alpha = x + y\sqrt{-3}, \quad x,y\in\ZZ,\quad x+y\equiv 1\pmod{2},\quad x\equiv 1\pmod{3}.
\]
These conditions can be always achieved by multiplication by a unity. Then we set
\[
c_{\pm}(\mathfrak{a}) = (-1)^{(x\mp y-1)/2} \alpha^6.
\]
With this notation, the forms $\varphi_{K,c_{\pm}}$ can be written as
\begin{equation}\label{eq:heckechar8}
\varphi_{K,c_{\pm}}(z)
= \sum_{\mathfrak{a}} c_{\pm}(\mathfrak{a}) q^{\mathrm{Norm}(\mathfrak{a})}.
\end{equation}
The Galois orbit $\{\varphi_{K,c_{+}},\varphi_{K,c_{-}}\}$
is listed with label 144.7.g.d on \cite{lmfdb}.

Next, by \Cref{1411:eq001}, we have
\[
\eta(12z)^{14} \equiv \sum_{n=0}^{+\infty} b_8(n) q^{24n+7} \pmod{2}.
\]
It follows that
\begin{equation}\label{eq:b8an}
b_8(n) \equiv \frac{a_{+}(24n+7)-a_{-}(24n+7)}{720\sqrt{-3}} \pmod{2},
\end{equation}
where $a_{\pm}(n)$ is the $n$-th Fourier coefficient of $\varphi_{K,c_{\pm}}(z)$.
By multiplicativity of the Fourier coefficients, see \Cref{eq:mult},
it suffices to study $a_{\pm}(n)$ on prime powers.

Our strategy is now to study $a_\pm(p^j)$ individually (in particular, their divisibility)
for primes dividing $24n+7$, later obtain the divisibility of $a_\pm(24n+7)$
by multiplicativity of the Fourier coefficients, and finally characterise the divisibility
of the difference $a_{+}(24n+7)-a_{-}(24n+7)$.
Note that, since the coefficients of $\eta$ and its powers are integers,
we must have \emph{a fortiori} that such a difference is an integer multiple of $720\sqrt{-3}$.
In particular, since $2^4||720$, we have that $2|b_8(n)$ if and only if
\[
\frac{a_{+}(24n+7)-a_{-}(24n+7)}{\sqrt{-3}}\equiv 0 \pmod{2^5},
\]
which is what we are going to characterise.

\begin{lemma}\label{lemma:511}
Let $a_\pm(n)$ be the $n$-th Fourier coefficient of $\varphi_{K,c_{\pm}}$
and let $p$ be a prime such that $p\equiv 5,11\pmod{12}$. Then
\[ 
\begin{cases}
a_{+}(p^j)=a_{-}(p^j)\equiv 1 \pmod{2},  &\text{if $j$ is even}, \\
a_{+}(p^j)=a_{-}(p^j)= 0,  &\text{if $j$ is odd}.
\end{cases}
\]
\end{lemma}

\begin{proof}
Since $p\equiv 2\pmod{3}$, $p$ is inert in $\Ocal_K$.
If $j$ is odd, there are no ideals of norm~$p^j$ and
therefore $a_{\pm}(p^j)=0$ by \Cref{eq:heckechar8}.
If $j$ is even, the only ideal of norm~$p^j$ is $p^{j/2}\Ocal_K$,
with generator $\pm p^{j/2}$, where the sign is chosen to have
the desired congruence modulo~3. By \Cref{eq:heckechar8}, this gives
$a_{+}(p^j)=a_{-}(p^j) \equiv p^{3j} \equiv 1\pmod{2}$, as claimed.
\end{proof}

\begin{lemma}\label{lemma:1mod12}
Let $a_\pm(n)$ be the $n$-th Fourier coefficient of $\varphi_{K,c_{\pm}}$
and let $p$ be a prime such that $p\equiv 1\pmod{12}$. Then
\[
a_{+}(p^j) = a_{-}(p^j) \equiv
\begin{cases}
1 \pmod{2}, & \text{if $j$ is even},\\
0 \pmod{2}, & \text{if $j$ is odd}.
\end{cases}
\]
\end{lemma}

\begin{proof}
Write $p=z^2+3w^2$, with $z,w\in\ZZ$. Reducing modulo 4, we deduce that
$z$ is odd and $w$ is even, so $z+w\equiv z-w\pmod{4}$. Therefore, we have $(-1)^{(z+w-1)/2}=(-1)^{(z-w-1)/2}$. There are two ideals above $p$, namely $(z+w\sqrt{-3})$ and $(z-w\sqrt{-3})$. Applying \Cref{eq:heckechar8}, we obtain
\[
a_{+}(p)=a_{-}(p)=
\begin{cases}
(z+w\sqrt{-3})^6 + (z-w\sqrt{-3})^6, & \text{if }z\pm w\equiv 1\pmod{4},\\
-(z+w\sqrt{-3})^6 - (z-w\sqrt{-3})^6, & \text{if }z\pm w\equiv 3\pmod{4}.\\
\end{cases}
\]
By \Cref{0810:tpbeta}, it follows that $a_{+}(p^j)=a_{-}(p^j)$ for all $j\in\ZZ_{\ge 0}$. Moreover, expanding the powers and reducing modulo 4, we obtain
\[
a_{\pm}(p) = \pm(2z^6 - 90z^4w^2 + 270z^2w^4 - 54w^6) \equiv 2 \pmod{4}.
\]
We also have, by \Cref{0810:tpbeta},
\[
a_\pm(p^j) \equiv \sum_{r=0}^{\lfloor j/2\rfloor} (-1)^r\binom{j-r}{r}2^{j-2r}
\equiv
\begin{cases}
(-1)^{j/2} \pmod{4}, & \text{if $j$ is even},\\
2 \pmod{4}, & \text{if $j\equiv 1\pmod{4}$},\\
0 \pmod{4}, & \text{if $j\equiv 3\pmod{4}$},
\end{cases}
\]
which implies the claim.
\end{proof}

\begin{lemma}\label{lemma:7mod12}
Let $a_\pm(n)$ be the $n$-th Fourier coefficient of $\varphi_{K,c_{\pm}}$
and let $p$ be a prime such that $p\equiv 7\pmod{12}$. If $j$ is even, then
\[
a_{+}(p^j) = a_{-}(p^j) \equiv (-1)^{j/2}\pmod{2^4}.
\]
If $j$ is odd, then
\[
a_{+}(p^j) = -a_{-}(p^j) = t\sqrt{-3},
\]
with
\[
t\equiv
\begin{cases}
2^3\pmod{2^4}, & \text{if $j\equiv 1\pmod{4}$},\\
0\pmod{2^4}, & \text{if $j\equiv 3\pmod{4}$}.
\end{cases}
\]
\end{lemma}
\begin{proof}
Write $p=z^2+3w^2$, with $z$ even and $w$ odd.
Up to replacing $w$ with $-w$, we assume that $z+w\equiv 1\pmod{4}$
and, hence, $z-w\equiv 3\pmod{4}$.
Applying \Cref{eq:heckechar8}, we obtain
\[
a_{+}(p) = -(z+w\sqrt{-3})^6 + (z-w\sqrt{-3})^6 = -a_{-}(p).
\]
Expanding the powers, we deduce that
\[
\frac{a_{+}(p)}{\sqrt{-3}} = -12z^5w + 120z^3w^3 - 108zw^5 \equiv -108zw^5 \equiv 2^3 \pmod{2^4},
\]
If $j$ is odd, by \Cref{0810:tpbeta} we obtain $a_{+}(p^j) = -a_{-}(p^j)$ and
\[
\frac{a_{+}(p^j)}{\sqrt{-3}} = \frac{1}{\sqrt{-3}}
\sum_{r=0}^{\lfloor j/2\rfloor} (-1)^r\binom{j-r}{r} p^{6r} a_{+}(p)^{j-2r} \equiv (j+1) 2^2 \pmod{2^4}.
\]
If $j$ is even, by \Cref{0810:tpbeta}, we have $a_{+}(p^j) = a_{-}(p^j)$ and
\[
a_{+}(p^j) = \sum_{r=0}^{\lfloor j/2\rfloor} (-1)^r\binom{j-r}{r} p^{6r} a_{+}(p)^{j-2r} \equiv (-1)^{j/2}p^{3j} \equiv (-1)^{j/2} \pmod{2^4}.
\]
\end{proof}

\begin{proof}[Proof of \Cref{thm:b8}]
Let
\[
24n+7 = \prod_{p} p^{\alpha_p}
\]
be the prime factorisation of the odd number $24n+7$. Hence $p=2$ never occurs.
Since the Fourier coefficients are multiplicative, see \Cref{eq:mult}, we have
\begin{equation}\label{0810:eq002}
a_\pm(24n+7) = \prod_p a_\pm(p^{\alpha_p}).
\end{equation}
By Lemmas \ref{lemma:511}, \ref{lemma:1mod12} and \ref{lemma:7mod12},
$a_{+}(p^j)=a_{-}(p^j)$ for every $j\in\ZZ_{\ge 0}$
and every prime $p\equiv 1,5,11\pmod{12}$
or when $j$ is even and $p\equiv 7\pmod{12}$.
On the other hand, by \Cref{lemma:7mod12}, we have
$a_{+}(p^j)=-a_{-}(p^j)$ when $j$ is odd and $p\equiv 7\pmod{12}$.
Therefore, we can write
\begin{equation}\label{0810:eq003}
a_{+}(24n+7) - a_{-}(24n+7) = (1-(-1)^\gamma)\prod_{p} a_{+}(p^{\alpha_p}),
\end{equation}
with
\[
\gamma = \sum_{\substack{p\equiv 7(12) \\ \alpha_p \text{ odd}}} \alpha_p.
\]
If $\gamma$ is even, then \Cref{0810:eq003} vanishes and $b_8(n)$ is even by \Cref{eq:b8an}.

Assume that $\gamma$ is odd, i.e., there is an odd number of primes
$p\equiv 7\pmod{12}$ appearing with an odd power.
By \Cref{lemma:7mod12}, each of such powers is of the form $t\sqrt{-3}$, with $2^3|t$.
It follows that if there are at least three of such powers we have
\begin{equation}\label{0810:eq004}
\frac{a_{+}(24n+7) - a_{-}(24n+7)}{\sqrt{-3}}\equiv 0 \pmod{2^5}
\end{equation}
and $b_8(n)$ is even by \Cref{eq:b8an}.

Assume thus that there is exactly one odd power $p^{\alpha_p}$ with $p\equiv 7\pmod{12}$.
Since the first factor in \Cref{0810:eq003} equals 2, by \Cref{lemma:7mod12} we deduce that
\Cref{0810:eq004} holds if $\alpha_p\equiv 3\pmod{4}$.

We are left with the case that exactly one prime $p\equiv 7\pmod{12}$ divides $24n+7$
with $\alpha_p\equiv 1\pmod{4}$.
In this case, by \Cref{lemma:7mod12} we have
\begin{equation}\label{0810:eq005}
\frac{a_{+}(24n+7) - a_{-}(24n+7)}{\sqrt{-3}}\equiv 2^4\prod_{p\not\equiv7(12)} a_{+}(p^{\alpha_p})\pmod{2^5}.
\end{equation}
We observe that, by \Cref{lemma:511} and \Cref{lemma:1mod12},
the last product in \Cref{0810:eq005} is even if there is a prime $p\not\equiv 7\pmod{12}$
appearing with odd power, and is odd otherwise.
This concludes the proof of the theorem by \Cref{eq:b8an}.
\end{proof}


\end{document}